\documentclass{article}
\usepackage{authblk}
\usepackage{abstract}
\usepackage{amsmath}
\usepackage{wasysym}
\usepackage{amsthm}
\usepackage{amssymb}

\newtheorem{Proposition}{Proposition}
\newtheorem{Lemma}{Lemma}
\newtheorem{Theorem}{Theorem}
\newtheorem{Remark}{Remark}
\newtheorem{Example}{Example}
\newtheorem{Corollary}{Corollary}

\begin{document}
	\title{Some aspects of generalized Dunkl-Williams constant in Banach spaces}
	\author[1]{Haoyu Zhou}
	\author[1,2]{Qi Liu\thanks{Qi Liu:liuq67@aqnu.edu.cn }}
	\author[1]{Yuxin Wang}
	
	\affil[1]{School of Mathematics and physics, Anqing Normal University, Anqing 246133,P.R.China}
	\affil[2]{International Joint Research Center of Simulation and Control for Population Ecology of Yantze River in Anhui ,Anqing Normal University, Anqing 246133,P.R.China}
	\maketitle 
	\begin{abstract}
	This article delves into an exploration of two innovative constants, namely $DW(X,\alpha,\beta)$ and $DW_B(X,\alpha,\beta)$, both of which constitute extensions of the Dunkl-Williams constant. We derive both the upper and lower bounds for these two constants and establish two equivalent relations between them. Moreover, we elucidate the relationships between these constants and several well-known constants. Additionally, we have refined the value of the $DW_B(X,\alpha,\beta)$ constant in certain specific Banach spaces.
	\end{abstract}
	\textbf{keywords:} {Dunkl–Williams constant; James constant;  Birkhoff orthogonality }
	
	\section{Introduction and preliminaries} 
	Throughout the paper, we always suppose that $X$ is a real Banach space with $di m X \geq 2$
	unless specifically stated otherwise, $B_X$ is the unit ball of $X$ and $S_X$ is the unit sphere
	of $X$.

	In 1964, C.F. Dunkl and K.S. Williams [1] showed that, in any Banach space $X$ with norm $\|\cdot\|,$ the inequality
	$$\left\|\frac{1}{\|x\|}x-\frac{1}{\|y\|}y\right\|\leqslant\frac{4}{\|x\|+\|y\|}\|x-y\|$$                        
	holds for all $x,y\in X$ with $x\neq0$ and $y\neq0$. Actually, the Dunkl-Williams inequality gives the upper bound for the angular distance 
	$$\alpha[x,y]:=\left\|\frac{x}{\|x\|}-\frac{y}{\|y\|}\right\|$$
	between two nonzero elements $x$ and $y$. The concept of angular distance was first introduced by Clarkson [2]. Further, in [1], Dunkl and Williams also found that if $X$ is a Hilbert space,
	then the Dunkl-Williams	inequality can be improved to the following inequality
	$$\left\|\frac{x}{\|x\|}-\frac{y}{\|y\|}\right\|\leq\frac{2\|x-y\|}{\|x\|+\|y\|},$$
	which holds for all nonzero elements $x$ and $y$. Soon after, in the same year that the Dunkl-Williams inequality came out, Kirk and Smiley [3] proved that the inequality in fact characterizes the Hilbert space. 
	
	According to the above results, Jimenez-Melado et al.[4] pointed out that the smallest number which 
	can replace $ 4 $ in Dunkl-Williams inequality actually measures the closeness between this Banach space and
	Hilbert space. Thus, Jimenez-Melado et al.[4] considered the Dunkl-Williams constant as following:
	$$DW(X)=\sup\left\{\frac{\|x\|+\|y\|}{\|x-y\|}\left\|\frac{x}{\|x\|}-\frac{y}{\|y\|}\right\|:x,y\in X\backslash\{0\},x\neq y\right\}.$$
	Based on the results of $DW(X)$, some constants are defined using other elements related to the upper bound of angular distance. In this regard, Massera and  Sch\"{a}ffer have proven  $\text{the Massera-Sch\"{a}ffer inequality}$ [5]. That is 
	$$\left\|\frac{x}{\|x\|}-\frac{y}{\|y\|}\right\|\leq\frac{2\|x-y\|}{\max\{\|x\|,\|y\|\}}$$
	holds for all nonzero elements $x$ and $y$. Al-Rashed [8] introduced the following parameter 
	$$\Psi_{\infty}(X)=\sup\left\{\frac{\max\{\|x\|,\|y\|\}}{\|x-y\|}\left\|\frac{x}{\|x\|}-\frac{y}{\|y\|}\right\|:x,y\in X\backslash\{0\},x\neq y\right\}.$$
	However, Baronti and Papini [9] proved that $\Psi_{\infty}(X)=2$ holds for any Banach space $X$, in other words, the Massera-Sch\"{a}ffer inequality is always sharp in any Banach space $X$.
	
	Let $x, y$ be two elements in a real Banach space $X$. Then $x$ is said to be Birkhoff orthogonal to $y$ and
	denoted by $x\perp_By$ [16], if
	$$\|x+\lambda y\|\geq\|x\|,\lambda\in\mathbb{R}.$$
	In addition, $x$ is said to be isosceles orthogonal to $y$ and denoted by $x\perp_Iy$ [17], if
	$$\|x+y\|=\|x-y\|.$$
	Recall that $x$ is said to be Singer orthogonal to $y$ and denoted by $x\perp_Sy$ [18], if either $\|x\|\cdot\|y\|=0$ or$$\left\|\frac{x}{\|x\|}-\frac{y}{\|y\|}\right\|=\left\|\frac{x}{\|x\|}+\frac{y}{\|y\|}\right\|.$$

Recently, quantitative studies of the difference between three orthogonality types have been performed:
	
	$$DW_S(X)=\sup\left\{\frac{\|x\|+\|y\|}{\|x-y\|}\left\|\frac{x}{\|x\|}-\frac{y}{\|y\|}\right\|:x,y\in X\backslash\{0\},x\perp_Sy\right\} ,$$

	$$DW_I(X)=\sup\left\{\frac{\|x\|+\|y\|}{\|x-y\|}\left\|\frac{x}{\|x\|}-\frac{y}{\|y\|}\right\|:x,y\in X\backslash\{0\},x\perp_Iy\right\},$$

	$$MS_B(X)=\sup\left\{\frac{\max\{\|x\|,\|y\|\}}{\|x-y\|}\left\|\frac{x}{\|x\|}-\frac{y}{\|y\|}\right\|:x,y\in X\backslash\{0\},x\perp_By\right\},$$
 $$DW_B(X)=\sup\left\{\frac{\|x\|+\|y\|}{\|x-y\|}\left\|\frac{x}{\|x\|}-\frac{y}{\|y\|}\right\|:x,y\in X\backslash\{0\},x\perp_By\right\}$$

	(see [13,14,15]).

	Recall that the modulus of convexity of $X$ is the function $\delta_{X}:[0,2]\to[0,1]$ give by [11]
	$$\begin{aligned}\delta_X(\varepsilon)&=\inf\left\{1-\left\|\frac{1}{2}(x+y)\right\|{:}x,y\in B_X,\|x-y\|\geq\varepsilon\right\}\\&=\inf\left\{1-\left\|\frac{1}{2}(x+y)\right\|{:}x,y\in S_X,\|x-y\|\geq\varepsilon\right\},\end{aligned}$$
	 and the characteristic of convexity of $X$ is defined as the number 
$$\varepsilon_0(X):=\sup\{\varepsilon\in[0,2]: \delta_X(\varepsilon)=0\}.$$
	The James constant is defined as [12]
	$$J(X)=\sup\left\{\min\left(\lVert x+y\rVert,\lVert x-y\rVert\right){:}\lVert x\rVert\leq1,\lVert y\rVert\leq1\right\}.$$
	We say that $X$ is uniformly nonsquare if there exists $\delta>0$
	such that for any pair $x,y\in B_{X}$ we have either $\|x+y\|\leq\delta $, or $\|x-y\|\leq\delta $. It is easy to verify that the three conditions X is uniformly nonsquare, $\varepsilon_{0}(X)<2$ and $J(X)<2$ are equivalent.

	\section{The $DW(X,\alpha,\beta)$ constant }
We define a new constant $DW(X,\alpha,\beta)$ by generalizing the  Dunkl-Williams constant. Let $X$ be regarded as a Banach space. We first outline the following key definitions: $\alpha,\beta>0$

	$$DW(X,\alpha,\beta)=\sup\left\{\frac{\alpha\|x\|+\beta\|y\|}{\|x-y\|}\left\|\frac{x}{\|x\|}-\frac{y}{\|y\|}\right\|:x,y\in X\backslash\{0\},x\neq y\right\}.$$ 
\begin{Remark}
	If $\alpha=\beta=1$, then
	$$DW(X,1,1)=DW(X)=\sup\left\{\frac{\|x\|+\|y\|}{\|x-y\|}\left\|\frac{x}{\|x\|}-\frac{y}{\|y\|}\right\|:x,y\in X\backslash\{0\},x\neq y\right\}.$$
\end{Remark}

	\begin{Proposition}
		Let $X$ be a Banach space, then $\alpha+\beta\leq DW(X,\alpha,\beta)\leq2(\alpha+\beta)$.
		
	\end{Proposition}
	\begin{proof}
	Let $y=-x$, then clearly
	 $$\frac{\alpha\vert\vert x\vert\vert+\beta\vert\vert y\vert\vert}{\vert\vert x-y\vert\vert}\bigg|\bigg|\frac{x}{\vert\vert x\vert\vert}-\frac{y}{\vert\vert y\vert\vert}\bigg\vert\bigg\vert=\alpha+\beta,$$
	which means that 
$DW(X,\alpha,\beta)\geq\alpha+\beta.$
	
		On the other hand, due to $\text{the Massera-Sch\"{a}ffer inequality}$[5]
	$$\left\|\frac{x}{\|x\|}-\frac{y}{\|y\|}\right\|\leq\frac{2\|x-y\|}{\max\{\|x\|,\|y\|\}},$$
	we have 
	$$\frac{\alpha\vert\vert x\vert\vert+\beta\vert\vert y\vert\vert}{\vert\vert x-y\vert\vert}\bigg\vert\bigg\vert\frac{x}{\vert\vert x\vert\vert}-\frac{y}{\vert\vert y\vert\vert}\bigg\vert\bigg\vert\leq\frac{2(\alpha\vert\vert x\vert\vert+\beta\vert\vert y\vert\vert)}{\max\{\vert\vert x\vert\vert,\vert\vert y\vert\vert\}}.$$

 We need to discuss it in two cases.
	
	Case 1 :  If $||x||\geq||y||$, then
		$$	\frac{2(\alpha||x||+\beta||y||)}{\max\{||x||,||y||\}}\leq\frac{2(\alpha||x||+\beta||x||)}{||x||}=2(\alpha+\beta).$$
	Case 2: If $||y||\geq||x||$ ,then
		$$	\frac{2(\alpha||x||+\beta||y||)}{\max\{||x||,||y||\}}\leq\frac{2(\alpha||y||+\beta||y||)}{||y||}=2(\alpha+\beta).$$
	Thus,we obtain that $$DW{(X,\alpha,\beta)}\leq2(\alpha+\beta).$$
	\end{proof}

		\begin{Proposition}
		Let $X$ be a Banach space. Then following gives the equivalent definition of the $DW(X,\alpha,\beta)$ constant.
		$$(1)DW(X,\alpha,\beta)=\sup\left\{\frac{\|x+y\|}{\bigg|\bigg|\frac{1}{\alpha}(1-\beta t)x+ty\bigg|\bigg|}:x,y\in S_X,0< t<\frac{1}{\beta}\right\}.$$
	$$(2)DW(X,\alpha,\beta)=\sup\left\{\frac{\|u+v\|}{\displaystyle\min_{0< t<\frac{1}{\beta}}\bigg|\bigg|\frac{1}{\alpha}(1-\beta t)u+tv\bigg|\bigg|}:u,v\in S_X\right\}.$$
	\end{Proposition}
	\begin{proof}
		First, for any $x,y\in X\backslash\{0\}$. Let $ u=\frac{x}{\|x\|},v=-\frac{y}{\|y\|}.$  Then, we have
		$$\begin{aligned}&\frac{\alpha\|x\|+\beta\|y\|}{\|x-y\|}\left\|\frac{x}{\|x\|}-\frac{y}{\|y\|}\right\|\\&=\frac{\|u+v\|}{\left\|\frac{\|x\|}{\alpha\|x\|+\beta\|y\|}u+\frac{\|y\|}{\alpha\|x\|+\beta\|y\|}v\right\|}\\&\leq\sup\left\{\frac{\|x+y\|}{\bigg|\bigg|\frac{1}{\alpha}(1-\beta t)x+ty\bigg|\bigg|}:x,y\in S_X,0< t<\frac{1}{\beta}\right\},\end{aligned}$$
		which implies that
		$$DW(X,\alpha,\beta)\leq\sup\left\{\frac{\|x+y\|}{\bigg|\bigg|\frac{1}{\alpha}(1-\beta t)x+ty\bigg|\bigg|}:x,y\in S_X,0< t<\frac{1}{\beta}\right\}.$$
		When $0< t<\frac{1}{\beta}$. Let $x=\frac{1}{\alpha}(1-\beta t)u\neq0,y=-tv\neq0.$
		$$\frac{\|u+v\|}{\|\frac{1}{\alpha}(1-\beta t)u+tv\|}=\frac{\alpha\|x\|+\beta\|y\|}{\|x-y\|}\left\|\frac{x}{\|x\|}-\frac{y}{\|y\|}\right\|\leq DW(X,\alpha,\beta).$$
		We obtain
		$$DW(X,\alpha,\beta)\geq\sup\left\{\frac{\|x+y\|}{\bigg|\bigg|\frac{1}{\alpha}(1-\beta t)x+ty\bigg|\bigg|}:x,y\in S_X,0< t<\frac{1}{\beta}\right\}.$$

		(2) By (1), it is evident that  
		$$DW(X,\alpha,\beta)\leq\sup\left\{\frac{\|u+v\|}{\displaystyle\min_{0< t<\frac{1}{\beta}}\bigg|\bigg|\frac{1}{\alpha}(1-\beta t)u+tv\bigg|\bigg|}:u,v\in S_X\right\}.$$
		For inverse inequality, since, for $u,v\in S_X,$ we must have 
		$$\min_{0< t<\frac{1}{\beta}}\bigg|\bigg|\frac{1}{\alpha}(1-\beta t)u+tv\bigg|\bigg|=\bigg|\bigg|\frac{1}{\alpha}(1-\beta t_0)u+t_0v\bigg|\bigg|$$
		for some $t_0\in(0,\frac{1}{\beta}),$ then, by using (1) again, we obtain
		$$DW(X,\alpha,\beta)\geq\sup\left\{\frac{\|u+v\|}{\displaystyle\min_{0< t<\frac{1}{\beta}}\bigg|\bigg|\frac{1}{\alpha}(1-\beta t)u+tv\bigg|\bigg|}:u,v\in S_X\right\}.$$
		This completes the proof.
	\end{proof}

	Next, we will present Lemma \ref{T1}. The technical means of the proof come from reference [10].
	\begin{Lemma}\label{T1}
		Let $X$ be a Banach space. If $(f_n)$ is sequence in $B_{X^{*}}$ and $(x_n)$ is a sequence in $B_X$ such that $\displaystyle\lim_{n\to\infty}f_{n}(x_{n})=1 $, then, for any sequence  $(g_n)$ in $B_{X^{*}}$  with $\lim\displaystyle\inf_{n\to\infty}g_{n}(x_{n})>0$, we have
		$$\begin{aligned}DW(X,\alpha,\beta)&\geq(\alpha+\beta)\max\left\{\operatorname*{liminf}_{n\to\infty}\left\|g_n(x_n)f_n-g_n\right\|,1\right\}\\&\geq(\alpha+\beta)\max\left\{\operatorname*{liminf}_{n\to\infty}g_n(x_n)\|f_n-g_n\|,1\right\}.\end{aligned}$$
	\end{Lemma}
	 \begin{proof}
	In the first place, we will show that
	$$DW(X,\alpha,\beta)\geq(\alpha+\beta)\max\left\{\operatorname*{liminf}_{n\to\infty}\|g_{n}(x_{n})f_{n}-g_{n}\|,1\right\}.$$
	If $\lim\displaystyle\inf_{n\to\infty}\|g_n(x_n)f_n-g_n\|\leq1$ , the inequality is immediately satisfied given that   $DW(X,\alpha,\beta)\geq\alpha+\beta$. Hence, assume that $\lim\displaystyle\inf_{n\to\infty}\|g_n(x_n)f_n-g_n\|>1$.
	Give $\varepsilon\in(1,\lim\displaystyle\inf_{n\to\infty}\|g_{n}(x_{n})f_{n}-g_{n}\|)$, there exists $n_{0}\geq1$ such that, for all $n\geq n_{0}$, the inequality $\|g_n(x_n)f_n-g_n\|>\varepsilon$ holds and we can then find $y_n\in S_X$ such that $(g_n(x_n)f_n-g_n)(y_n)>\varepsilon.$
	Let $\mathrm{t>0}$ and let us define, for each $n\geq n_{0}$, $z_{n}=x_{n}+ty_{n}$.
	By definition of $DW(X,\alpha,\beta)$, we have, for each $n\geq n_{0}$,
	$$\begin{aligned}DW(X,\alpha,\beta)&\geq\frac{\alpha\|x_n\|+\beta\|z_n\|}{\|x_n-z_n\|}\left\|\frac{x_n}{\|x_n\|}-\frac{z_n}{\|z_n\|}\right\|\\&=\frac{1}{t}\left(\frac{\alpha\|x_n\|}{\|z_n\|}+\beta\right)\left\|\frac{\|z_n\|}{\|x_n\|}x_n-z_n\right\|.\end{aligned}$$
	Since $(X_{n})$ is a sequence in $B_X$ and $ \displaystyle\operatorname*{lim}_{n\to\infty}f_{n}(x_{n})=1$, it must be $\displaystyle\operatorname*{lim}_{n\to\infty}\|x_{n}\|=1$ and therefore
	$$\begin{aligned}DW(X,\alpha,\beta)&\geq\frac{1}{t}\left(\frac{\alpha}{\displaystyle\lim\sup_{n\to\infty}\|z_{n}\|}+\beta\right)\operatorname*{liminf}_{n\to\infty}\left\|\frac{\|z_{n}\|}{\|x_{n}\|}x_{n}-z_{n}\right\|\\&=\frac{1}{t}\left(\frac{\alpha}{\displaystyle\lim\sup_{n\to\infty}\|z_{n}\|}+\beta\right)\operatorname*{lim}_{n\to\infty}\inf\left\|\|z_{n}\|x_{n}-z_{n}\right\|.\end{aligned}$$
	Moreover, for each $n\geq n_{0}$, we have
	$$\begin{Vmatrix}\|z_n\|x_n-z_n\end{Vmatrix}=\begin{Vmatrix}(\|z_n\|-1)x_n-ty_n\end{Vmatrix}\geq\left(\|z_n\|-1\right)g_n(x_n)+tg_n(-y_n),$$
	and, in addition,
	$$\|z_n\|=\|x_n+ty_n\|\geq f_n(x_n)+tf_n(y_n),$$
	so that
	$$\begin{aligned}\|\|z_{n}\|x_{n}-z_{n}\|&\geq\left(f_{n}(x_{n})+tf_{n}(y_{n})-1\right)g_{n}(x_{n})+tg_{n}(-y_{n})\\&\geq-\left(1-f_{n}(x_{n})\right)+t\varepsilon.\end{aligned}$$
	Therefore
	$$\lim_{n\to\infty}\inf\left\|\|z_n\|x_n-z_n\right\|\geq t\varepsilon,$$
	and in consequence
	$$\begin{aligned}DW(X,\alpha,\beta)&\geq\left(\frac{\alpha}{\displaystyle\lim\sup_{n\to\infty}\lVert z_n\rVert}+\beta\right)\varepsilon\\&=\left(\frac{\alpha}{\displaystyle\lim\sup_{n\to\infty}\lVert x_n+ty_n\rVert}+\beta\right)\varepsilon\\&\geq\left(\frac{\alpha}{1+t}+\beta\right)\varepsilon.\end{aligned}$$
	Letting $t\rightarrow0^{+}$, we obtain $DW(X,\alpha,\beta)\geq(\alpha+\beta)\epsilon$.
	We have proved that, for any $\varepsilon\in(1,\operatorname*{lim}\displaystyle\operatorname*{inf}_{n\to\infty}\|g_{n}(x_{n})f_{n}-g_{n}\|)$, the inequality $DW(X,\alpha,\beta)\geq(\alpha+\beta)\epsilon$ holds. Thus 
	$$DW(X,\alpha,\beta)\geq(\alpha+\beta)\lim_{n\to\infty}\inf\left\|g_n\left(x_n\right)f_n-g_n\right\|.$$
   Now, based on the proof provided in reference [10], we can confirm that the following inequality holds.
	$$\max\left\{\lim_{n\to\infty}\inf\left\|g_{n}(x_{n})f_{n}-g_{n}\right\|,1\right\}\geq\max\left\{\lim_{n\to\infty}\inf g_{n}(x_{n})\|f_{n}-g_{n}\|,1\right\}.$$
   This completes the proof.
	   
	\end{proof}

	\begin{Theorem}\label{t2}
	For every Banach space $\text{X}$, $DW(X,\alpha,\beta)\geq(\alpha+\beta)\max\{\varepsilon_{0}(X),1\}.$
	\end{Theorem}
	\begin{proof}
    Since we know that $DW(X,\alpha,\beta)\geq(\alpha+\beta)$
	, we only need to show that $DW(X,\alpha,\beta)\geq(\alpha+\beta)\varepsilon_{0}(X).$
	If $\varepsilon_{0}(X)\leq1$ nothing needs to be proved. Otherwise, there exist two sequences $\{u_{n}\}$ and $\{v_{n}\}$ in $S_X$ such that $\|u_{n}-v_{n}\|\to\varepsilon_{0}(X)$ and $\|u_n+v_n\|\to2$
	Consider, for each $n\geq1,f_{n},g_{n}\in S_{X^{*}}$ such that $f_n(u_n+v_n)=\|u_n+v_n\|$ and $g_n(u_n-v_n)=\|u_n-v_n\|$. Observe that 
	$$\lim_{n\to\infty}f_n(u_n)=\lim_{n\to\infty}f_n(v_n)=1,$$
	since
	$$\lim_{n\to\infty}\left(f_n(u_n)+f_n(v_n)\right)=\lim_{n\to\infty}\lVert u_n+v_n\rVert=2,$$
	
and $|f_{n}(u_{n})|\leq1,|f_{n}(v_{n})|\leq1.$
	
	In addition,
	$$\begin{aligned}&\operatorname*{liminf}_{n\to\infty}g_n(u_n)\\&=\operatorname*{liminf}_{n\to\infty}\left(\lVert u_n-v_n\rVert+g_n(v_n)\right)\\&\geq\varepsilon_0(X)-1\\&>0,\end{aligned}$$
	and hence, by Lemma \ref{T1}, we can obtain that
	$$\begin{aligned}DW(X,\alpha,\beta)& \geq (\alpha+\beta)\liminf_{n \to \infty}\|g_n(u_n)f_n - g_n(v_n)\|\\& \geq (\alpha+\beta)\liminf_{n \to \infty}(g_n(u_n)f_n - g_n(v_n)) \\& = (\alpha+\beta)\lim_{n \to \infty}g_n(u_n - v_n) \\& = (\alpha+\beta)\lim_{n \to \infty}\|\mu_n - v_n\|\\& = (\alpha+\beta)\varepsilon_0(X).\end{aligned}$$
	\end{proof}
	
For the lower bound, we have established a direct connection between $DW(X,\alpha,\beta)$ and $\varepsilon_{0}(X)$. Next, we will establish a connection between $DW(X,\alpha,\beta)$ and $J(X)$.

	\begin{Theorem}\label{t1}
	 For any Banach space $X$ we have that
	 $$\begin{aligned}DW(X,\alpha,\beta)&\leq\sup_{0\leq t\leq2}\min\left\{2(\alpha+\beta)-\frac{\alpha+\beta}{2}\delta_X(t),(\alpha+\beta)+\frac{\alpha+\beta}{2}t\right\}\\&=(\alpha+\beta)+\frac{\alpha+\beta}{2}J(X).\end{aligned}$$
	\end{Theorem}
	\begin{proof}

		Let $x,y\in X$ with $x\neq0,y\neq0,x-y\neq0$ .
		Using the triangle inequality
		$$\begin{aligned}&\left\|{\frac{\alpha\|x\|+\beta\|y\|}{\|x\|}}x-{\frac{\alpha\|x\|+\beta\|y\|}{\|y\|}}y\right\|\|x-y\|^{-1}\\&\leq\left\|\alpha x-\alpha{\frac{\|x\|}{\|y\|}}y\right\|\|x-y\|^{-1}+\left\|\beta{\frac{\|y\|}{\|x\|}}x-\beta y\right\|\|x-y\|^{-1}\\&=\alpha\left\|\frac{x-y}{\|x-y\|}+\frac{y-\frac{\|x\|}{\|y\|}y}{\|x-y\|}\right\|+\beta\left\|\frac{\frac{\|y\|}{\|x\|}x-x}{\|x-y\|}+\frac{x-y}{\|x-y\|}\right\|.\end{aligned}$$
		By the definition of $\delta_{X}$,
			$$\begin{aligned}&\left\|{\frac{\alpha\|x\|+\beta\|y\|}{\|x\|}}x-{\frac{\alpha\|x\|+\beta\|y\|}{\|y\|}}y\right\|\|x-y\|^{-1}\\&\leq2\alpha\left(1-\delta_{X}\left(\frac{\|x+y(\frac{\|x\|}{\|y\|}-2)\|}{\|x-y\|}\right)\right)+2\beta\left(1-\delta_{X}\left(\frac{\|y+x(\frac{\|y\|}{\|x\|}-2)\|}{\|x-y\|}\right)\right)\\&\leq2\alpha\left(1-\delta_{X}\left(\frac{|\|x\|-|\|x\|-2\|y\|||}{\|x-y\|}\right)\right)+2\beta\left(1-\delta_{X}\left(\frac{|\|y\|-|\|y\|-2\|x\|||}{\|x-y\|}\right)\right).
		\end{aligned}$$
	From the above relation it is straightforward to obtain that
	$$\begin{aligned}&\left\|{\frac{\alpha\|x\|+\beta\|y\|}{\|x\|}}x-{\frac{\alpha\|x\|+\beta\|y\|}{\|y\|}}y\right\|\|x-y\|^{-1}\\&\leq2(\alpha+\beta)-(\alpha+\beta)\delta_{X}\left({\frac{2|\|x\|-\|y\||}{\|x-y\|}}\right),\end{aligned}$$
	discuss separately the three possibilities: $\|y\|\leq\|x\|/2,\|x\|/2<\|y\|\leq2\|x\|\mathrm{~or~}\|y\|>2\|x\|.$
	On the other hand, using again the triangle inequality in 
	$$\begin{aligned}&\left\|{\frac{\alpha\|x\|+\beta\|y\|}{\|x\|}}x-{\frac{\alpha\|x\|+\beta\|y\|}{\|y\|}}y\right\|\|x-y\|^{-1}\\&\leq(\alpha+\beta)+(\alpha+\beta){\frac{|\|x\|-\|y\||}{\|x-y\|}}.\end{aligned}$$
	We have obtained two upper bounds for $\left\|{\frac{\alpha\|x\|+\beta\|y\|}{\|x\|}}x-{\frac{\alpha\|x\|+\beta\|y\|}{\|y\|}}y\right\|\|x-y\|^{-1}$, and consequently the following one,
	$$\begin{aligned}&\left\|{\frac{\alpha\|x\|+\beta\|y\|}{\|x\|}}x-{\frac{\alpha\|x\|+\beta\|y\|}{\|y\|}}y\right\|\|x-y\|^{-1}\\&\leq\min\left\{2(\alpha+\beta)-(\alpha+\beta)\delta_{X}\left(\frac{2\|x\|-\|y\||}{\|x-y\|}\right),(\alpha+\beta)+(\alpha+\beta)\frac{|\|x\|-\|y\||}{\|x-y\|}\right\}\\&\leq\sup_{0\leq t\leq2}\min\{2(\alpha+\beta)-(\alpha+\beta)\delta_{X}(t),(\alpha+\beta)+\frac{(\alpha+\beta)}{2}t\}.	\end{aligned}$$
	We conclude that
	$$\begin{aligned}DW(X,\alpha,\beta)&=\sup\left\{\left\|{\frac{\alpha\|x\|+\beta\|y\|}{\|x\|}}x-{\frac{\alpha\|x\|+\beta\|y\|}{\|y\|}}y\right\|\|x-y\|^{-1}{:}x\neq0,y\neq0,x-y\neq0\right\}\\&\leq\sup_{0\leq t\leq2}\min\left\{2(\alpha+\beta)-(\alpha+\beta)\delta_X(t),(\alpha+\beta)+\frac{(\alpha+\beta)}{2}t\right\},\end{aligned}$$
	as desired.
	
	Consider the function $f:[0,2]\to[2,4]$ defined by
	$$f(t)=\min\left\{2(\alpha+\beta)-(\alpha+\beta)\delta_X(t),(\alpha+\beta)+\frac{(\alpha+\beta)}{2}t\right\}.$$
	To complete the proof we have to show that
	$$\sup_{0\leq t\leq2}f(t)=(\alpha+\beta)+\frac{(\alpha+\beta)}{2}J(X).$$
	Observe that if $\varepsilon_{0}(X)=2$,or equivalently $ J(X)=2,$
	we have, for $0\leq t<2$
	$$2(\alpha+\beta)-(\alpha+\beta)\delta_X(t)=2(\alpha+\beta)>(\alpha+\beta)+\frac{(\alpha+\beta)}{2}t,$$
	and then
	$$\sup_{0\leq t\leq2}f(t)=2(\alpha+\beta)=(\alpha+\beta)+\frac{(\alpha+\beta)}{2}J(X).$$
	Otherwise, i.e. if $\varepsilon_{0}(X)<2$ , the continuity of $\delta_{X}$ in $[0,2)$ gives the existence of a solution to the equation
	$$2(\alpha+\beta)-(\alpha+\beta)\delta_X(t)=(\alpha+\beta)+\frac{(\alpha+\beta)}{2}t,$$
	in the interval $[\varepsilon_0(X),2)$. Moreover, this solution is unique because $\phi_{1}(t)=2(\alpha+\beta)-(\alpha+\beta)\delta_{X}(t)$ is nonincreasing and $\phi_{2}(t)=(\alpha+\beta)+\frac{(\alpha+\beta)}{2}t$ is strictly increasing. If we denote this solution by $t_{X} $, it is clear that
	$$\begin{aligned}t_X&=\sup\left\{t\in[0,2] : 2(\alpha+\beta)-(\alpha+\beta)\delta_X(t)>(\alpha+\beta)+\frac{(\alpha+\beta)}{2}t\right\}\\&=\sup\left\{t\in[0,2] : \frac{(\alpha+\beta)}{2}-\frac{(\alpha+\beta)t}{4}>\frac{(\alpha+\beta)}{2}\delta_X(t)\right\},\end{aligned}$$
	and also that $f$ attains its maximum value at $t_{X} $ because $f$ is increasing on $(0,t_{X})$, decreasing on $(t_{X},2)$ and continuous on $(0,2)$.We have then that
	$$\sup_{0\leq t\leq2}f(t)=(\alpha+\beta)+\frac{(\alpha+\beta)}{2}t_X.$$
	On the other hand, it was proved in [20] that
	$$J(X)=\sup\left\{\varepsilon\in(0,2) : \delta_X(\varepsilon)\leq1-\frac{\varepsilon}{2}\right\},$$
	and thus $t_{X}=J(X)$, which finishes the proof.

	\end{proof}
	\begin{Corollary}\label{t3}
		For any Banach space $X$, we have that
		$$(\alpha+\beta)\max\begin{Bmatrix}\varepsilon_0(X),1\end{Bmatrix}\leq DW(X,\alpha,\beta)\leq(\alpha+\beta)+\frac{(\alpha+\beta)}{2}J(X).$$
		
	\end{Corollary}
	\begin{Example}
		For $\mu\geq1$ let $X_{\mu}$ be the space $\ell_{2}$ endowed with the norm
		$$|x|_\mu=\max \left\{\|x\|_2, \mu\|x\|_{\infty}\right\}.$$
		The space  $X_{\mu}$ have been extensively studied because they play a major role in metric fixed point theory. It is well known that.
$$\varepsilon_0(X_\mu)=\left\{\begin{array}{cc}2(\mu^2-1)^{\frac{1}{2}},&\mu\leq\sqrt{2},\\2,&\mu\geq\sqrt{2},\end{array}\right.$$
and it was also shown in [6] that 
$$J(X_\mu)=\min\{2,\mu\sqrt{2}\}.$$
In particular, for $1<\mu<\sqrt{2}$ , we have that $J(X_\mu)=\mu\sqrt{2}.$  Therefore the above corollary yields
$$2(\alpha+\beta)(\mu^2-1)^{\frac{1}{2}}\leq DW(X_\mu,\alpha,\beta)\leq(\alpha+\beta)+\frac{(\alpha+\beta)}{2}\mu\sqrt{2},$$
provided that $1\leq\mu\leq\sqrt{2}$.
	\end{Example}
	 \begin{Theorem}
	 	Let $X$ be a Banach space with $DW(X,\alpha,\beta)<2(\alpha+\beta)$ and let $Y$ be a Banach space isomorphic to $X$. Then 
	 	$$DW(Y,\alpha,\beta)\leq(\alpha+\beta)+\frac{(\alpha+\beta)}{2}J(X)d(X,Y).$$
	 	\end{Theorem}
	 \begin{proof}
	 	It was shown in [6] that $J(Y)\leq J(X)d(X,Y)$, and then the result follows from Corollary \ref{t3}.
	 \end{proof}
	
	\begin{Corollary}\label{t4}
	 Suppose that $X$ is a Hilbert space and that $Y$ is a Banach space isomorphic to $X$. Then 
	 $DW(Y,\alpha,\beta)\leq(\alpha+\beta)+\sqrt{2}\frac{(\alpha+\beta)}{2}J(X)d(X,Y).$  In particular, if $d(X,Y)<\sqrt{2},$ then $DW(Y,\alpha,\beta)<2(\alpha+\beta).$
	\end{Corollary}
	 \begin{proof}
	 	It is a particular case of Theorem 3 taking into account that $J(X)=\sqrt{2}.$
	 \end{proof}
	 
	Recall that the Lindenstrauss modulus of smoothness is the function $\rho_X:[0,\infty)\to\mathbb{R}$ given by [19]
	$$\rho_X(t)=\sup\left\{\frac{1}{2}\left(\|x+ty\|+\|x-ty\|\right)-1 : x,y\in B_X\right\}.$$
	The coefficient
	$$\rho_X^{\prime}(0)=\lim_{t\to0^+}\frac{\rho_X(t)}{t}$$
	is often called the characteristic of smoothness of $X$. The following theorem relates $DW(X,\alpha,\beta)$ and the characteristic of smoothness of $X$.
	\begin{Theorem}
	In any Banach space $X$, the inequality $DW(X,\alpha,\beta)\geq(\alpha+\beta)\max\{2\rho_{X}^{\prime}(0),1\}$ holds.
		\end{Theorem}
		\begin{proof}
The inequality $DW(X,\alpha,\beta)\geq(\alpha+\beta)$ always holds. We have then to prove that $DW(X,\alpha,\beta)\geq2(\alpha+\beta)\rho_{X}^{\prime}(0).$ If ${\mathrm{DW}}(X,\alpha,\beta)=2(\alpha+\beta),$ the
inequality is obvious, so we can assume  ${\mathrm{DW}}(X,\alpha,\beta)<2(\alpha+\beta),$ and then the reflexivity of $X$.

Let ${\varepsilon}\in[0,2]$ such that $\delta_{X^*}(\varepsilon)=0.$ For such $\varepsilon$ there exist two sequences $(f_n)$ and $(g_n)$
in $S_{X^{*}}$ such that $\|f_n-g_n\|=\varepsilon$ for all $n\geq1$ and $\displaystyle\lim_{n\to\infty}\|f_n+g_n\|=2.$ Consider, for each $n\geq1$ , $x_n\in S_{X}$ such that $(f_n+g_n)(x_n)=\|f_n+g_n\|$. It must be
$$\lim_{n\to\infty}f_n(x_n)=\lim_{n\to\infty}g_n(x_n)=1.$$
By  Lemma \ref{T1}, we have
$$DW(X,\alpha,\beta)\geq(\alpha+\beta)\lim_{n\to\infty}g_n(x_n)\|f_n-g_n\|=(\alpha+\beta)\varepsilon.$$
We have proved that, for any $\varepsilon\in[0,2]$ such that $\delta_{X^{*}}(\varepsilon)=0,$ we have $DW(X,\alpha,\beta)\geq(\alpha+\beta)\varepsilon.$ Therefore 
$$\begin{aligned}DW(X,\alpha,\beta)&\geq(\alpha+\beta)\sup\left\{\varepsilon\in[0,2]{:}\delta_{X^*}(\varepsilon)=0\right\}\\&=(\alpha+\beta)\varepsilon_0(X^*)\\&=2(\alpha+\beta)\rho_X^{\prime}(0).\end{aligned}$$
		\end{proof}

\section{The $DW_B(X,\alpha,\beta)$ constant }
In this section, based on the constant $DW(X,\alpha,\beta)$, we will perform some manipulations on $x$ and $y$ to obtain another new constant $DW_{B}(X,\alpha,\beta)$ and study some properties of this constant. We first outline the following key definition: $\alpha, \beta>0$

	$$DW_B(X,\alpha,\beta)=\sup\left\{\frac{\alpha\|x\|+\beta\|y\|}{\|x-y\|}\left\|\frac{x}{\|x\|}-\frac{y}{\|y\|}\right\|:x,y\in X\backslash\{0\},x \perp_B y\right\}.$$ 

\begin{Proposition}
	Let $X$ be a Banach space. Then 
	$$(\alpha+\beta)\leq DW_{B}(X,\alpha,\beta)\leq\max\{2\alpha+\beta,\alpha+2\beta\}.$$
\end{Proposition}		
\begin{proof}
	First, take $x,y\in X\backslash\{0\}$ with $x\perp_{B}y$, and let $u=\frac{x}{\|x\|},v=\frac{y}{\|y\|}.$  From the
	homogeneity of Birkhoff orthogonality, we obtain $u\perp_{B}v.$ Then
	$$DW_B(X,\alpha,\beta)\geq\frac{\alpha\|u\|+\beta\|v\|}{\|u-v\|}\left\|\frac{u}{\|u\|}-\frac{v}{\|v\|}\right\|=(\alpha+\beta).$$
	Second, to obtain an upper bound of 	$DW_B(X,\alpha,\beta)$ for any  $x,y\in X\backslash\{0\}$ with  $x\perp_{B}y$, there are two cases that need to be considered separately.
	
	If $\|x\|\leq\|y\|$, we can obtain that

	$$\begin{aligned}\alpha\|x\|\left\|\frac{x}{\|x\|}-\frac{y}{\|y\|}\right\|&=\alpha\left\|\frac{\|x\|}{\|y\|}(x-y)+\left(1-\frac{\|x\|}{\|y\|}\right)x\right\|\\&\leq\alpha\frac{\|x\|}{\|y\|}\|x-y\|+\alpha\left(1-\frac{\|x\|}{\|y\|}\right)\|x-y\|\\&=\alpha\|x-y\|,\end{aligned}$$
	and
	$$\begin{aligned}\beta\|y\|\left\|{\frac{x}{\|x\|}}-{\frac{y}{\|y\|}}\right\|&\leq\beta\|y\|\left\|\frac{x}{\|x\|}-\frac{x}{\|y\|}\right\|+\beta\|y\|\left\|\frac{x}{\|y\|}-\frac{y}{\|y\|}\right\|\\&=\beta\|y\|-\beta\|x\|+\beta\|x-y\|\\&\leq2\beta\|x-y\|.\end{aligned}$$
	Hence, we obtain
	$$\frac{\alpha\|x\|+\beta\|y\|}{\|x-y\|}\left\|\frac{x}{\|x\|}-\frac{y}{\|y\|}\right\|\leq(\alpha+2\beta).$$
	In fact, if $\|y\|\leq\|x\|$(for $\|x\|\leq\|y\|$ the proof is similar), then
	
	$$\frac{\alpha\|x\|+\beta\|y\|}{\|x-y\|}\left\|\frac{x}{\|x\|}-\frac{y}{\|y\|}\right\|\leq(2\alpha+\beta).$$
	
	Thus, we deduce $$DW_B{(X,\alpha,\beta)}\leq\max{\{2\alpha+\beta,\alpha+2\beta\}}.
	$$
	
\end{proof}		
\begin{Example}
	Let $X=(\mathbb{R}^2,\|\cdot\|_\infty).$ Then $DW_B{(X,\alpha,\beta)}=\max{\{2\alpha+\beta,\alpha+2\beta\}}
.	$
\end{Example}
\begin{proof}
	Let $x=\left(\frac{1}{3},\frac{1}{3}\right),y=\left(0,\frac{2}{3}\right).$ Then, straightforward calculations show that $x\perp_{B}y.$ Thus
	$$DW_B(X,\alpha,\beta)\geq\frac{\alpha\|x\|+\beta\|y\|}{\|x-y\|}\left\|\frac{x}{\|x\|}-\frac{y}{\|y\|}\right\|=\alpha+2\beta.$$
Let $x=\left(0,\frac{2}{3}\right),y=\left(\frac{1}{3},\frac{1}{3}\right),$ we obtain
		$$DW_B(X,\alpha,\beta)\geq\frac{\alpha\|x\|+\beta\|y\|}{\|x-y\|}\left\|\frac{x}{\|x\|}-\frac{y}{\|y\|}\right\|=2\alpha+\beta,$$
as desired.
	\end{proof}

\begin{Proposition}
	Let $X$ be a Banach space. Then the following gives the equivalent definition of the $DW_B(X,\alpha,\beta)$ constant.
	$$(1)DW_B(X,\alpha,\beta)=\sup\left\{\frac{\|u+v\|}{\bigg|\bigg|\frac{1}{\alpha}(1-\beta t)u+tv\bigg|\bigg|}:u,v\in S_X,u\perp_Bv,0< t<\frac{1}{\beta}\right\}.$$
	$$(2)DW_B(X,\alpha,\beta)=\sup\left\{\frac{\|u+v\|}{\displaystyle\min_{0< t<\frac{1}{\beta}}\bigg|\bigg|\frac{1}{\alpha}(1-\beta t)u+tv\bigg|\bigg|}:u,v\in S_X,u\perp_Bv\right\}.$$
\end{Proposition}
\begin{proof}
(1) First, for any $x,y\in X\backslash\{0\}$ with $ x\perp_{B}y,$
let $u=\frac{x}{\|x\|},v=-\frac{y}{\|y\|}.$ Then we
have 
$$\frac{\alpha\|x\|+\beta\|y\|}{\|x-y\|}\left\|\frac{x}{\|x\|}-\frac{y}{\|y\|}\right\|=\frac{\|u+v\|}{\left\|\frac{\|x\|}{\alpha\|x\|+\beta\|y\|}u+\frac{\|y\|}{\alpha\|x\|+\beta\|y\|}v\right\|}.$$
Since the Birkhoff orthogonality is homogeneous, we obtain $u\perp_{B}v$ . Then, due to, we obtain 
$$DW_B(X,\alpha,\beta)\leq\sup\left\{\frac{\|u+v\|}{\bigg|\bigg|\frac{1}{\alpha}(1-\beta t)u+tv\bigg|\bigg|}:u,v\in S_X,u\perp_Bv,0< t<\frac{1}{\beta}\right\}.$$
Second, let $u,v\in S_X$ with $u\perp_{B}v$. If $0<t<\frac{1}{\beta}$, let $x=\frac{1}{\alpha}(1-\beta t)u\neq0,y=-tv\neq0.$ Then, $x\perp_{B}y$ and 
$$\frac{\|u+v\|}{\bigg|\bigg|\frac{1}{\alpha}(1-\beta t)u+tv\bigg|\bigg|}=\frac{\alpha\|x\|+\beta\|y\|}{\|x-y\|}\left\|\frac{x}{\|x\|}-\frac{y}{\|y\|}\right\|\leq DW_B(X,\alpha,\beta).$$
Consequently, we obtain
$$\sup\left\{\frac{\|u+v\|}{\bigg|\bigg|\frac{1}{\alpha}(1-\beta t)u+tv\bigg|\bigg|}:u,v\in S_X,u\perp_Bv,0< t<\frac{1}{\beta}\right\}\leq DW_B(X,\alpha,\beta).$$
(2) By (1), it is evident that  
$$DW_B(X,\alpha,\beta)\leq\sup\left\{\frac{\|u+v\|}{\displaystyle\min_{0< t<\frac{1}{\beta}}\bigg|\bigg|\frac{1}{\alpha}(1-\beta t)u+tv\bigg|\bigg|}:u,v\in S_X,u\perp_Bv\right\}.$$
For inverse inequality, since, for $u,v\in S_X,$ we must have 
$$\min_{0< t<\frac{1}{\beta}}\bigg|\bigg|\frac{1}{\alpha}(1-\beta t)u+tv\bigg|\bigg|=\bigg|\bigg|\frac{1}{\alpha}(1-\beta t_0)u+t_0v\bigg|\bigg|$$
for some $t_0\in(0,\frac{1}{\beta}),$ then, by using (1) again, we obtain
$$DW_B(X,\alpha,\beta)\geq\sup\left\{\frac{\|u+v\|}{\displaystyle\min_{0< t<\frac{1}{\beta}}\bigg|\bigg|\frac{1}{\alpha}(1-\beta t)u+tv\bigg|\bigg|}:u,v\in S_X,u\perp_Bv\right\}.$$
This completes the proof.

\end{proof}		
	Recall that the rectangular constant $\mu(X)$ introduced by Joly [7] is defined as follows:
		$$\mu(X)=\sup\left\{\frac{\|x\|+\|y\|}{\|x+y\|}:x,y\in X\backslash\{0\},x\perp_By\right\}.$$
The following theorem establishes the relation between $DW_B(X,\alpha,\beta)$ and $ \mu(X).$		
		
\begin{Theorem}
Let $X$ be a Banach space. Then
$$\min\{\alpha,\beta\}\mu(X)\leq DW_{B}(X,\alpha,\beta)\leq2\max\{\alpha,\beta\}\mu(X).$$
\end{Theorem}		
\begin{proof}
Since the Birkhoff orthogonality is homogeneous, one can easily deduce that 
	$$\mu(X)=\sup\left\{\frac{\|x\|+\|y\|}{\|x-y\|}:x,y\in X\backslash\{0\},x\perp_By\right\}.$$
In the first place, we will show that 
	$$DW_{B}(X,\alpha,\beta)\geq\min\{\alpha,\beta\}\mu(X).$$
	Since, for any $ x, y$ belong to $ X\backslash\{0\}$ with $x\perp_{B}y$, then we have $\frac{x}{||x||}\perp_{B} \frac{y}{||y||}$. Hence, we can obtain that the following inequality .
	$$\begin{aligned}\frac{\alpha||x||+\beta||y||}{||x-y||}\bigg|\bigg|\frac{x}{||x||}-\frac{y}{||y||}\bigg|\bigg|&\geq\frac{\alpha||x||+\beta||y||}{||x-y||}\\&\geq\min\{\alpha,\beta\}\frac{||x||+||y||}{||x-y||},\end{aligned}$$

which means that	$DW_{B}(X,\alpha,\beta)\geq\min\{\alpha,\beta\}\mu(X).$
	
To prove the right inequality:
	$$\begin{aligned}\frac{\alpha||x||+\beta||y||}{||x-y||}\bigg|\bigg|\frac{x}{||x||}-\frac{y}{||y||}\bigg|\bigg|&\leq2\frac{(\alpha||x||+\beta||y||)}{||x-y||}\\&\leq2\max\{\alpha,\beta\}\frac{||x||+||y||}{||x-y||}.\end{aligned}$$

Thus, we can obtain 
	$$DW_{B}(X,\alpha,\beta)\leq2\max\{\alpha,\beta\}\mu(X).$$
	
	This completes the proof.
\end{proof}

\end{document}